\makeatletter
\def\input@path{{macros/}}
\makeatother

\documentclass[final,acmtoms]{acmsmall}

\usepackage{url,etex}
\usepackage[utf8x]{inputenc}

\usepackage{graphics,graphicx}
\graphicspath{{./figure/}}
\usepackage{epstopdf,float}
\usepackage{amsfonts}
\usepackage[boxruled,longend,linesnumbered]{algorithm2e}

\newtheorem{problem}{Problem}

\usepackage[english]{babel}

\usepackage{graphics,graphicx}
\usepackage{bm,bbm,cases}

\usepackage[math]{easyeqn}
\usepackage{easyvector}

\eqrowsep{2pt plus 1pt minus 1pt}
\eqspacing{4pt plus 2pt minus 1pt}

\newvector[\mathit{0},\bm{0}]{zero}
\newvector[\mathit{a},\bm{a}]{avec}
\newvector[\mathit{b},\bm{b}]{bvec}
\newvector[\mathit{c},\bm{c}]{cvec}
\newvector[\mathit{d},\bm{d}]{dvec}
\newvector[\mathit{p},\bm{p}]{pvec}

\newcommand{\ds}{\,\mathrm{d}s}

\newcommand{\dxi}{\,\mathrm{d}\xi}

\newcommand{\dz}{\,\mathrm{d}z}

\newcommand{\dtau}{\,\mathrm{d}\tau}

\newcommand{\Cf}{\mathcal{C}}
\newcommand{\Sf}{\mathcal{S}}

\renewcommand{\rmdefault}{bch}
\newcommand{\reffun}[1]{\texttt{\ref{#1}}}
\IncMargin{0.5em}

\newcommand{\DEF}{\mathrel{\mathop:}=}
\newcommand{\assign}{\leftarrow}

\newcommand{\ASSIGNl}[1]{\rlap{$#1$}\qquad\assign\,}
\newcommand{\ASSIGNm}[1]{\rlap{$#1$}\quad\;\assign\,}
\newcommand{\ASSIGNs}[1]{\rlap{$#1$}\quad\assign\,}

\begin{document}

\markboth{E.Bertolazzi and M.Frego}{Fast and accurate clothoid fitting}

\title{Fast and accurate clothoid fitting}

\author{
\uppercase{Enrico Bertolazzi} and \uppercase{Marco Frego}\affil{University of Trento, Italy}}

\begin{abstract}
An effective solution to the problem of Hermite $G^1$ interpolation 
with a clothoid curve is provided.
At the beginning the problem is naturally formulated as a system of nonlinear equations
with multiple solutions that is generally difficult to solve numerically.
All the solutions of this nonlinear system are reduced to the computation
of the zeros of a single nonlinear equation.
A simple strategy, together with the use of a good and simple guess function, 
permits to solve the single nonlinear equation with a few iterations of the
Newton--Raphson method.

The computation of the clothoid curve requires the computation of Fresnel 
and Fresnel related integrals.
Such integrals need asymptotic expansions 
near critical values to avoid loss of precision.
This is necessary when, for example, the solution of interpolation problem 
is close to a straight line or an arc of circle.
Moreover, some special recurrences are deduced for the efficient computation
of asymptotic expansion.

The reduction of the problem to a single nonlinear function in one variable
and the use of asymptotic expansions make the solution algorithm fast and robust.

\end{abstract}




\acmformat{Enrico Bertolazzi, Marco Frego. Fast and accurate clothoid fitting}

\begin{bottomstuff}
Author's addresses: Department of Mechanical and Structural Engineering
\end{bottomstuff}

\maketitle

\section{Introduction}

There are several curves proposed
for Computer Aided Design~\cite{Farin:2001,Baran:2010,deBoor:1978},
for planning trajectories of robots and vehicles 
or the geometric design of roads \cite{dececco:2007,Scheuer:1997}.
The most important among these are clothoids also known as Euler's or Cornu's spirals,
clothoids splines, i.e. a planar curve consisting in clothoid segments, circles and straight lines,
\cite{Davis:1999,Meek:1992,Meek:2004,Meek:2009,Walton:2009},
generalized clothoids or Bezier spirals \cite{Walton:1996}.
Pythagorean Hodograph \cite{Walton:2007,Farouki:1995}, bi-arcs and conic curves are also widely 
used~\cite{Pavlidis:1983}.
It is well known that the best or most pleasing curve is the clothoid, despite its transcendental form.

The procedure that allows a plane curve to interpolate two given points 
with assigned (unit) tangent vectors is called $G^1$ \emph{Hermite interpolation},
if given curvatures at the two points is also satisfied, 
then this is called $G^2$ Hermite interpolation \cite{McCrae:2008}.
A single clothoid segment is not enough to ensure $G^2$ Hermite interpolation,
because of the insufficient degrees of freedom.
Many authors have provided algorithms that use a \emph{pair} of clothoids 
segments to reach the $G^2$ Hermite interpolation requests.
However, often it is considered enough the cost-effectiveness of a $G^1$
Hermite interpolation, because it has been seen numerically that 
the discontinuity of the curvature is negligible.\\
The purpose of this paper is to describe a new method for
$G^1$ Hermite interpolation with a single clothoid segment,
which does not need to split the problem in mutually exclusive 
cases as in \citeNP{Meek:2009}, and does not suffer in case of 
degenerate Hermite data like straight lines or circles (see Figure~\ref{fig:1and2} on the right). 
These are of course limiting cases but are treated naturally in our approach. 
Because of numerical stability in the computation,
we preferred to introduce an appropriate threshold to avoid 
such degenerate situation in practice.
A precise estimate of such switching points is discussed.
Finally, the problem of the $G^1$ Hermite interpolation 
is reduced to a solution of a single nonlinear equation,
for example by the Newton--Raphson method. 
In order to provide a fast and accurate algorithm, we give a good initial guess for the 
Newton--Raphson method solver so that a few iterations suffice. \\ 
The remainder of the article is structured as follows. 
In the next section we define the interpolation problem, in the third we describe the passages to reformulate it such that from three equations in three unknowns  it reduces to one nonlinear equation in one unknown. This is enough from the theoretical point of view. The fourth section is devoted to the discussion of a good starting point for the Newton--Raphson method, so that using that guess we achieve a quick convergence, and how to select a correct solution of the nonlinear equation. 
Section \ref{sec:5} introduces the Fresnel integrals in their standard form and shows how the $G^1$ interpolation problem (i.e. the argument of the trigonometric functions involved is not purely a second order monomial, but a complete second order polynomial) can be conducted to that form.  Section \ref{sec:6} analizes what happens when the parameters of computation give numerical instabilities in the formulas, and \textit{ad hoc} expressions for such cases are provided. In the appendix there are the algorithms written in pseudo-code and a summary of the algorithm for the accurate computation of the clothoid spline.

\section{The fitting problem}
The construction of highway and railway routes and the 
trajectories of mobile robots can be split in the in construction 
of a piecewise clothoid curve which definition is given next.
\begin{definition}[Clothoid curve]
The general parametric form of a clothoid spiral curve is the following
\begin{EQ}[rcl]\label{clot}
  x(s) &=& x_0 + \int_0^s \cos \left(\frac{1}{2}\kappa'\tau^2+\kappa\tau+\vartheta_0\right) \dtau, \\
  y(s) &=& y_0 + \int_0^s \sin \left(\frac{1}{2}\kappa'\tau^2+\kappa\tau+\vartheta_0\right) \dtau,
\end{EQ}
where $s$ is the arc length, 
$\kappa't+\kappa$ is the
linearly varying curvature, $(x_0,y_0)$ is the starting point and $\vartheta_0$
initial angle.
Notice that $\frac{1}{2}\kappa's^2+\kappa s+\vartheta_0$
is the angle of the curve at arc length $s$.
\end{definition}
Clothoids curves are computed via the Fresnel sine and cosine integrals
(see \citeNP{abramowitz:1964} for a definition)
which is discussed forward in Section \ref{sec:5}.

The determination of the parameters $\vartheta_0$, $\kappa$ and $\kappa'$
are determined by points and angles at the extrema of the curve.
Thus, the problem considered in this paper is stated next.
\begin{problem}\label{prob:1}
  Given two points $(x_0,y_0)$ and $(x_1,y_1)$
  and two angles $\vartheta_0$ and $\vartheta_1$, 
  find a clothoid segment of the form~\eqref{clot} which
  satisfies:
  \begin{EQ}\label{eq:prob:1}
    x(0) = x_0, \qquad
    y(0) = y_0, \qquad
    \arctan\left(\dfrac{y'(0)}{x'(0)}\right) = \vartheta_0,\qquad \\
    x(L) = x_1, \qquad
    y(L) = y_1, \qquad
    \arctan\left(\dfrac{y'(L)}{x'(L)}\right) = \vartheta_1,\qquad \\
  \end{EQ}
  with minimal positive $L$ (the length of the curve). 
  The general scheme is showed in Figure~\ref{fig:1and2} on the left.
\end{problem}
\begin{figure}[!bt]
  \begin{center}
    \includegraphics[scale=0.7]{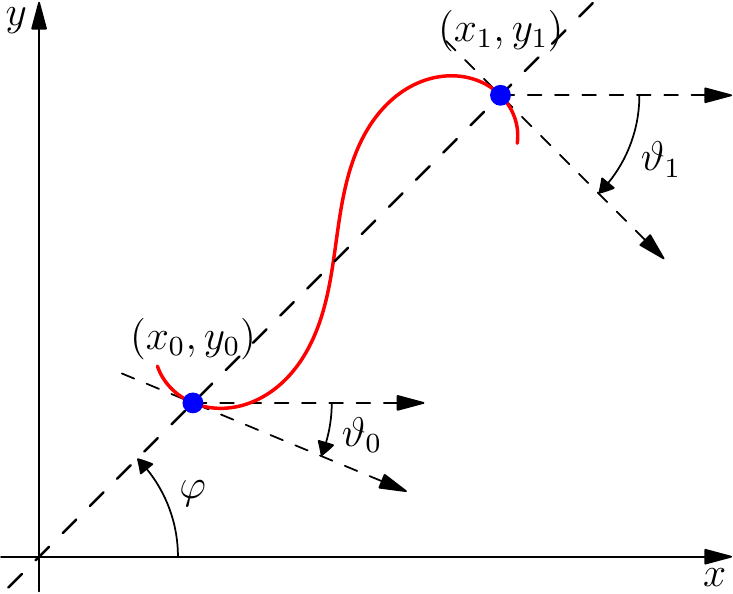}
    \qquad
    \includegraphics[scale=0.7]{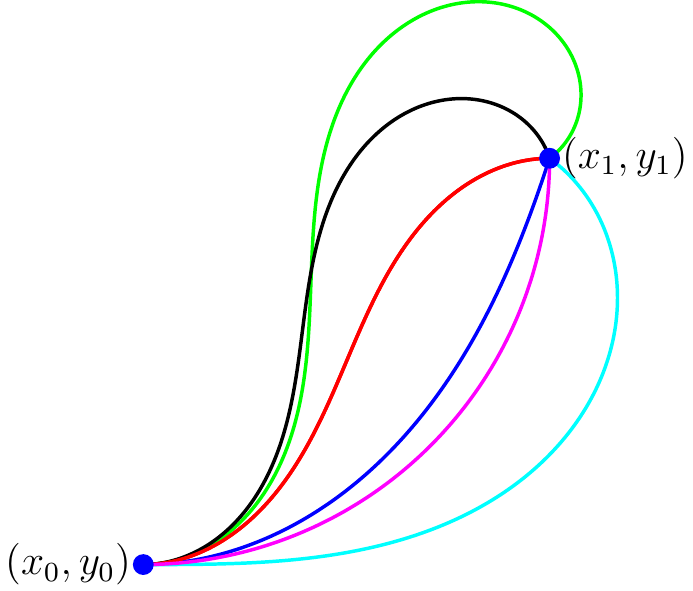}
  \end{center}
  \caption{On the left $G^1$ Hermite interpolation schema. On the right some cases
           representing different final angles.
           The magenta arc is the limiting case of a circle.
           The other colored arcs represent intermediate cases.}
  \label{fig:1and2}
\end{figure}
Solution of Problem~\ref{prob:1} is a zero of the following
nonlinear system involving the unknowns $L$, $\kappa$, $\kappa'$:
\begin{EQ}\label{eq:F:nonlin}
   \bm{F}(L,\kappa, \kappa')
   =
   \pmatrix{
     x_1-x_0 -\int_0^L \cos \left(\frac{1}{2}\kappa's^2+\kappa s+\vartheta_0\right) \ds\\
     y_1-y_0 -\int_0^L \sin \left(\frac{1}{2}\kappa's^2+\kappa s+\vartheta_0\right) \ds\\
     \vartheta_1-\left(\frac{1}{2}\kappa'L^2+\kappa L+\vartheta_0\right)
   }.\qquad
\end{EQ}
Find the points such that $\bm{F}(L,\kappa, \kappa')=\zero$
is a difficult problem to solve in this form.
In the next section a reformulation of the problem in a simpler
form permits to solve it easily. 

\section{Reformulation of the problem}
\label{sec:3}
Nonlinear system~\eqref{eq:F:nonlin} is put in an equivalent 
form by introducing the parametrization $s=\tau L$ so that
the integrals involved in computation have fixed extrema independent of $L$:
\begin{EQ}[rcl]\label{eq:F:nonlin:2}
   \bm{F}\left(L,\dfrac{B}{L},\dfrac{2\,A}{L^2}\right)
   &=&
   \pmatrix{
     \Delta x -L\int_0^1 \cos \left(A\tau^2+B\tau+\vartheta_0\right) \dtau \\
     \Delta y -L\int_0^1 \sin \left(A\tau^2+B\tau+\vartheta_0\right) \dtau \\
     \vartheta_1-\left(A+B+\vartheta_0\right)
   },\qquad
\end{EQ}
where $A=\frac{1}{2}\kappa' L^2$, $B=L\,\kappa$, $\Delta x = x_1-x_0$, $\Delta y = y_1-y_0$.\\
The third equation in \eqref{eq:F:nonlin:2} is linear so that we can 
solve it with respect to $B$,
\begin{EQ}\label{eq:F:nonlin:3:pre}
   B = \Delta\vartheta-A, \qquad \Delta\vartheta=\vartheta_1-\vartheta_0,
\end{EQ}
and the solution of nonlinear system \eqref{eq:F:nonlin:2} is reduced to the
solution of the nonlinear system of two equations in two unknown, namely $L$ and $A$:
\begin{EQ}\label{eq:F:nonlin:6}
   \bm{G}(L,A)
   =
   \pmatrix{
     \Delta x -L\int_0^1 \cos \left(A\tau^2+(\Delta\vartheta-A)\tau+\vartheta_0\right) \dtau \\
     \Delta y -L\int_0^1 \sin \left(A\tau^2+(\Delta\vartheta-A)\tau+\vartheta_0\right) \dtau \\
   },\qquad
\end{EQ}
followed by the computation of $B$ using \eqref{eq:F:nonlin:3:pre}.
Nonlinear system \eqref{eq:F:nonlin:6} is easier to solve than~\eqref{eq:F:nonlin:2}.
We can perform further simplification using polar coordinates
to represent $(\Delta x,\Delta y)$, namely
\begin{EQ}\label{eq:xhhyhh:polar}
   \Delta x = r\cos\varphi,\qquad
   \Delta y = r\sin\varphi,
\end{EQ}
and from~\eqref{eq:xhhyhh:polar} and $L>0$ we can define two new nonlinear
functions $f(L,A)$ and $g(A)$, where $g(A)$ is independent of $L$, 
as follows:
\begin{EQ}
   f(L,A)=\bm{G}(L,A)\cdot\pmatrix{\cos\varphi\\\sin\varphi},\qquad
   g(A)=\dfrac{1}{L}\bm{G}(L,A)\cdot\pmatrix{\sin\varphi\\-\cos\varphi}.
\end{EQ}
Using the identity $\sin(\alpha-\beta)=\sin\alpha\cos\beta-\cos\alpha\sin\beta$,
function $g(A)$ simplifies in:
\begin{EQ}\label{eq:g}
   g(A) = \Theta(A;\Delta\vartheta,\Delta\varphi),
\end{EQ}
where $\Delta\varphi=\vartheta_0-\varphi$ and 
\begin{EQ}\label{eq:def:Theta}
  \Theta(A;\Delta\vartheta,\Delta\varphi)
  = \int_0^1 
  \sin\big(A\tau^2+(\Delta\vartheta-A)\tau+\Delta\varphi\big)
  \dtau,
\end{EQ}
while using the identity $\cos(\alpha-\beta)=\cos\alpha\cos\beta+\sin\alpha\sin\beta$,
function $f(L,A)$ reduces to
\begin{EQ}[rcl]\label{eq:h}
  f(L,A)
  &=& 
  \sqrt{\Delta x^2+\Delta y^2}-L\,h(A), \\
  h(A) &=&\int_0^1 
  \cos\big(A\tau^2+(\Delta\vartheta-A)\tau+\Delta\varphi\big)
  \dtau
  = \Theta\left(A;\Delta\vartheta,\Delta\varphi+\frac{\pi}{2}\right).
\end{EQ}
\begin{remark}
  Defining $\phi_0=\vartheta_0-\varphi$ and $\phi_1=\vartheta_1-\varphi$
  we have $\Delta\vartheta = \phi_1-\phi_0$ and $\Delta\varphi= \phi_0$,
  moreover, the angles $\phi_0$ and $\phi_1$ are the same used in~\citeNP{Walton:2009}
  to derive $G^1$ interpolant.
\end{remark}

\begin{lemma}\label{lem:Theta:0}
  The function $\Theta$ defined in~\eqref{eq:def:Theta} satisfy
  \begin{EQ}\label{eq:thetanz}
    \Theta\left(A;\Delta\vartheta,z\right)=0,
    \quad\Rightarrow\quad
    \Theta\left(A;\Delta\vartheta,z+\frac{\pi}{2}\right)\neq 0.
  \end{EQ}
\end{lemma}
\begin{proof}
  If $A=0$ by a simple computation we obtain
  \begin{EQ}\label{eq:th:1}
    \Theta\left(A;\Delta\vartheta,z\right)
    =
    \cases{ \dfrac{\cos\Delta\vartheta-\cos(\Delta\vartheta+z)}{\Delta\vartheta}
    & if $\Delta\vartheta\neq 0$, \\
    \dfrac{1-\cos z}{z} & if $\Delta\vartheta =0$,
  }
  \end{EQ}
  and it is immediate to verify~\eqref{eq:thetanz}. 
  Let $A>0$, after some manipulation we have
  \begin{EQ}
    \dfrac{\sqrt{2A}}{\sqrt{\pi}}
    \Theta\left(A;\Delta\vartheta,z\right)
    = \big(\Sf(a)+\Sf(b)\big)\cos\eta
    - \big(\Cf(a)+\Cf(b)\big)\sin\eta 
  \end{EQ}
  where
  \begin{EQ}[c]\label{eq:th:2}
    a = \dfrac{A-\Delta\vartheta}{\sqrt{2\pi A}}, \qquad
    b = \dfrac{A+\Delta\vartheta}{\sqrt{2\pi A}}, \qquad
    \eta = \dfrac{\pi}{2}a^2-z.
  \end{EQ}
  If implication \eqref{eq:thetanz} is false, i.e. both $ \Theta\left(A;\Delta\vartheta,z\right)=0$ and 
  $\Theta\left(A;\Delta\vartheta,z+\pi/2\right)=0$
  then from previous equations it follows
  \begin{EQ}
    \Sf(a)+\Sf(b)=0, \qquad
    \Cf(a)+\Cf(b)=0,
  \end{EQ}
  and the symmetry $\Cf(-z)=-\Cf(z)$ with $\Sf(-z)=-\Sf(z)$ 
  implies that points $(\Cf(a),\Sf(a))$ and $(\Cf(-b),\Sf(-b))$ are coincident 
  points on the Cornu spiral 
  in parametric
  form $(\Cf(t),\Sf(t))$ and thus $a=-b$. Hence,
  \begin{EQ}
     0  = a+b = 
     \dfrac{A-\Delta\vartheta}{\sqrt{2\pi A}} + 
     \dfrac{A+\Delta\vartheta}{\sqrt{2\pi A}}
     = \dfrac{\sqrt{2\,A}}{\sqrt{\pi}},
  \end{EQ}
  and thus $A=0$ which contradict the assumption that $A>0$.
  The case with $A<0$ is similar.
  \qed
\end{proof}
Actually, the situation can be understood graphically looking at the plots of $g(A)$ and $h(A)$,
as in figure ~\ref{fig:3andCompare} on the left.
It is possible to reverse the implication by the following lemma.
\begin{lemma}\label{lem:LAB}
  The solutions of the nonlinear system \eqref{eq:F:nonlin:6}
  are given by
  \begin{EQ}\label{eq:LAB}
    L = \dfrac{\sqrt{\Delta x^2+\Delta y^2}}{h(A)},\qquad
    \kappa = \dfrac{\Delta\vartheta-A}{L}, \qquad
    \kappa' = \dfrac{2\,A}{L^2},\qquad
  \end{EQ}  
  where $A$ is a root of $g(A)$ defined in equation~\eqref{eq:g}
  and $h(A)$ is defined in~\eqref{eq:h}.
\end{lemma}
\begin{proof}
  Let $L$, $A$ satisfy~\eqref{eq:g} and \eqref{eq:LAB}.
  Then $f(L,A)=0$ and thus $\bm{G}(L,A)=\zero$.
  From Lemma~\ref{lem:Theta:0} when $g(A)=0$ then $h(A)\neq 0$ and
  thus $L$ is well defined.\qed
\end{proof}
\begin{remark}
  The solution of problem~\ref{prob:1} is build
  using Lemma~\ref{lem:LAB} selecting solution of minimal length.
  Notice that the parameter $A$ corresponding to minimal (positive) $L$
  depend only on the \emph{relative} angles and it is independent 
  on scaling factor.
  So that we can compute $A$ using only $\Delta\varphi$ and $\Delta\vartheta$
  assuming $r=\sqrt{\Delta x^2+\Delta y^2}=1$.
\end{remark}
Hence the interpolation problem is reduced 
to a single (nonlinear) equation that can be solved 
numerically with Newton--Raphson Method.
By a graphical inspection we notice that the roots
of $g(A)$ are simple, so that Newton--Raphson
converges quadratically.
Function \reffun{alg:findA} with \reffun{alg:buildGrid}
implements a simple strategy for the computation of minimum length 
clothoid on a grid of points and angles.
This algorithm is used to compute surface presented in 
Figure~\ref{fig:AL} on the top.

\begin{function}[H]
  \caption{findA($A_{\mathrm{guess}}$, $\Delta\vartheta$, $\Delta\varphi$,  $\mathrm{tol}$)}
  \label{alg:findA}
  \def\assign{\leftarrow}
  \SetAlgoLined 
  \small
  Define $g$ as $g(A):=\Theta(A,\Delta\vartheta,\Delta\varphi)$; Set $A\assign A_{\mathrm{guess}}$\;
  \lWhile{$\abs{g(A)}>\mathrm{tol}$}{
    $A\assign A - g(A)/g'(A)$
  }\;
  \Return{$A$}\;
\end{function}

\begin{function}[H]
  \caption{buildGrid($N$, $M$)}
  \small
  \label{alg:buildGrid}
  \SetKwFunction{findParams}{findParams}
  \def\assign{\leftarrow}
  \SetAlgoLined
  \lFor{$j=0,1,\ldots,M$}{$A^{\mathrm{guess}}_j\assign (j/M)40-20$}\;
  \For{$i,j=0,\ldots,N$}{
    $\Delta\varphi\assign 2\pi i/N-\pi$;\quad
    $\Delta\vartheta\assign 2\pi j/N-\pi$;\quad
    $L_{ij}\assign \infty$\;
    \For{$k=1,\ldots,M$}{
      \If{$\Theta(A^{\mathrm{guess}}_k,\Delta\vartheta,\Delta\varphi)\,
           \Theta(A^{\mathrm{guess}}_{k-1},\Delta\vartheta,\Delta\varphi)\leq 0$}{
        $\widetilde{A}\assign\,$\findA($(A^{\mathrm{guess}}_k+A^{\mathrm{guess}}_{k-1})/2$,$\Delta\vartheta$,$\Delta\varphi$)\;
        $\widetilde{L}\assign 1/\Theta(\widetilde{A},\Delta\vartheta,\Delta\varphi+\pi/2)$\;
        \lIf{$\widetilde{L}>0$ and $\widetilde{L}<L_{ij}$}{
          $L_{ij}\assign\widetilde{L}$;~ $A_{ij}\assign\widetilde{A}$\;
        }
      }
    }
  }
  \Return{$\bm{A}$, $\bm{L}$}\;
\end{function}
In the next section a recipe to select a good initial
point is described.
\begin{figure}
  \begin{center}
    \includegraphics[scale=0.65]{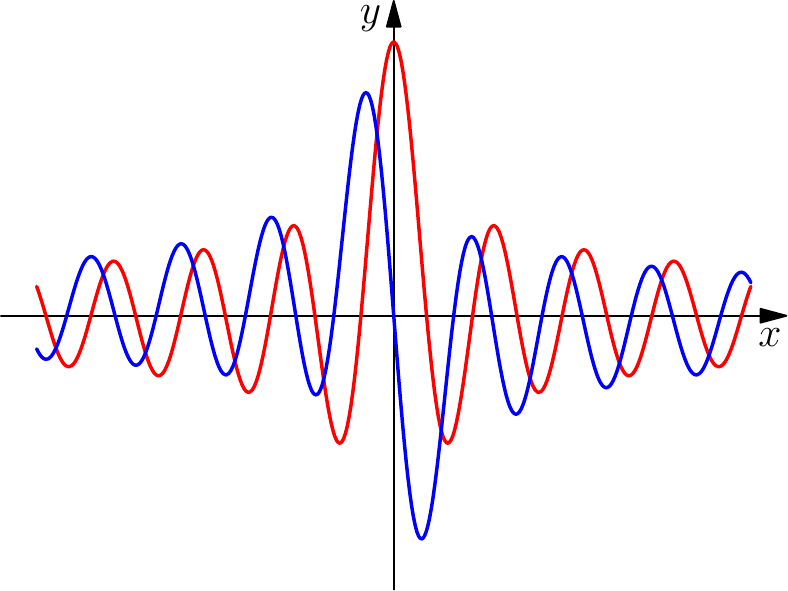}\quad
    \includegraphics[scale=0.95]{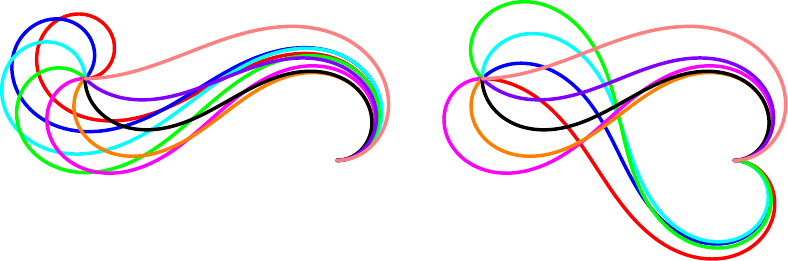}
  \end{center}
  \caption{On the left: plot of $g$ (red line) and $h$ (blue line), when $g(A)=0$ then $h(A)\neq 0$
           so that the length $L$ is well defined.
           In centre clothoids computed using~\eqref{eq:A:approx}
           as starting point for Newton-Raphson iterative scheme, on the
           right the minimum length clothoids. The clothoids start from 
           $(0,0)$ angle $0$ to $(-0.95,0.31)$ and angle ranging from $-\pi$ to $\pi$.}
  \label{fig:3andCompare}
\end{figure}

\section{Solution of $G^1$ Hermite interpolation problem}
\label{sec:4}
By a graphical inspection the solution is in the range 
$[-20,20]$ for angles $\Delta\varphi$ and $\Delta\vartheta$
in the interval $[-\pi,\pi]$.
The function~\reffun{alg:buildGrid} is used to compute 
all the possible solutions in this interval
for angles ranging in $[-\pi,\pi]$ and selecting the ones
with minimum length.
The strategy is very simple; the interval $[-20,20]$ is 
divided in small intervals and where $g(A)$ changes
sign Newton-Raphson iterative scheme is used.
If convergence is attained and $L$, computed using~\eqref{eq:LAB},
is positive, this solution is a candidate
to be the final solution.

This strategy, although simple, is effective and permits to evaluate
the solution of Problem~\ref{prob:1} at any point.
Figure~\ref{fig:AL} shows the results of computation using~\reffun{alg:buildGrid}
with $N=32$ and $M=128$.
Notice that the solution is discontinuous for higher and lower angles.
To avoid this jump in the solution we drop the requirement of 
minimal length for all angles and solve the following problem. 
\begin{problem}\label{prob:2}
  Given two points $(x_0,y_0)$ and $(x_1,y_1)$ and two angles 
  $\vartheta_0$ and $\vartheta_1$, find a clothoid segment of 
  the form~\eqref{clot} which satisfies~\eqref{eq:prob:1}
  with minimal positive $L$ for $\abs{\vartheta_0}$ and 
  $\abs{\vartheta_1}$ less then $\pi/2$.
  For greater angles we choose the solution that make
  the function $\mathcal{A}(\Delta\vartheta,\Delta\varphi)$
  with $\mathcal{A}(\Delta\vartheta,\Delta\varphi)\in\{A\,|\,\Theta(A,\Delta\vartheta,\Delta\varphi)=0\}$
  a continuos surface. 
\end{problem}
Looking at Figure~\ref{fig:AL} we guess that $\mathcal{A}(\Delta\vartheta,\Delta\varphi)$
can be approximated by a plane.
A simple fitting on the solution of Problem~\ref{prob:2} 
with a plane, results in the following approximation for $\mathcal{A}(\Delta\vartheta,\Delta\varphi)$:
\begin{EQ}\label{eq:A:approx}
  \mathcal{A}(\Delta\vartheta,\Delta\varphi) \approx
   2.4674\,\Delta\vartheta + 5.2478\,\Delta\varphi.
\end{EQ}
Using~\eqref{eq:A:approx} as starting point for Newton-Raphson 
the correct solution for Problem~\ref{prob:2} is found in few 
iterations. Figure~\ref{fig:AL} show the computed solution
found using this strategy. 
Computing the solution with Newton-Raphson starting with the guess
given by~\eqref{eq:A:approx} in a $1024\times 1024$ grid
with a tolerance of $10^{-10}$ results in the following iteration distribution:
\begin{center}
\begin{tabular}{c|ccccc}
  iter  & 1 & 2 & 3   & 4    & 5 \\
  \hline
  \#cases & 1 & 212 & 148880 & 846868 & 54664 \\
     & \multicolumn{2}{c}{less than  $1\%$} & $14\%$ & $80\%$ & $5\%$ 
\end{tabular}
\end{center}
Notice that the average number of iteration is $3.9$ while
the maximum number of iterations per point is $5$. 
Figure~\ref{fig:3andCompare} on the right shows the difference of computing minimum
length clothoid curve and clothoid curve of Problem~\ref{prob:2}.
The minimum length produces a bad looking solution for some
angles combinations. The curve produced by the solution of Problem~\ref{prob:2}
looks better and in any case curve produced by the solution of
Problem~\ref{prob:1} and Problem~\ref{prob:2} are identical
for a large combination of points and angles.

The algorithm for the clothoid computation is resumed
in function~\reffun{alg:buildClothoid} which uses 
\reffun{alg:normalizeAngle} to put angles in the correct
ranges.\\
\begin{function}[H]
  \caption{normalizeAngle($\varphi$)}
  \label{alg:normalizeAngle}
  \def\assign{\leftarrow}
  \SetAlgoLined 
  \small
  \lWhile{$\varphi>+\pi$}{$\varphi\assign\varphi-2\pi$}\;
  \lWhile{$\varphi<-\pi$}{$\varphi\assign\varphi+2\pi$}\;
  \Return{$\varphi$}\;
\end{function}
Function \reffun{alg:buildClothoid} solves equation \eqref{eq:g} that is it finds a solution of the equation
 $ g(A)=0$ by calling \reffun{alg:findA} described at the end of Section \ref{sec:3}.
\begin{function}[H]
  \caption{buildClothoid($x_0$, $y_0$, $\vartheta_0$, $x_1$, $y_1$, $\vartheta_1$)}
  \label{alg:buildClothoid}
  \SetKwFunction{findA}{findA}
  \SetKwFunction{Aguess}{Aguess}
  \SetKwFunction{normalizeAngle}{normalizeAngle}
  \def\assign{\leftarrow}
  \small
  \SetAlgoLined 
  $\Delta x \assign x_1-x_0$;\quad
  $\Delta y \assign y_1-y_0$\;
  Compute $r$ and $\varphi$ from $r\cos\varphi=\Delta x$ and $r\sin\varphi=\Delta y$\;
  $\ASSIGNm{\Delta\varphi}\normalizeAngle(\vartheta_0-\varphi)$;\quad
  $\ASSIGNm{\Delta\vartheta}\normalizeAngle(\vartheta_1-\vartheta_0)$\; 
  $\ASSIGNm{A}\findA( 2.4674\,\Delta\vartheta + 5.2478\,\Delta\varphi,\Delta\vartheta, \Delta\varphi)$\;
  $L\assign r/\Theta(A,\Delta\vartheta,\Delta\varphi+\frac{\pi}{2})$;~
  $\kappa\assign(\Delta\vartheta-A)/L$;~
  $\kappa'\assign(2\,A)/L^2$\;
  \Return{$\kappa$, $\kappa'$, $L$}
\end{function}
The algorithm needs accurate computation of Fresnel 
related functions $g(A)$ and $h(A)$ with relative
derivative which are discussed in the next section.
\begin{figure}[!htcb]
  \begin{center}
    \includegraphics[scale=0.45]{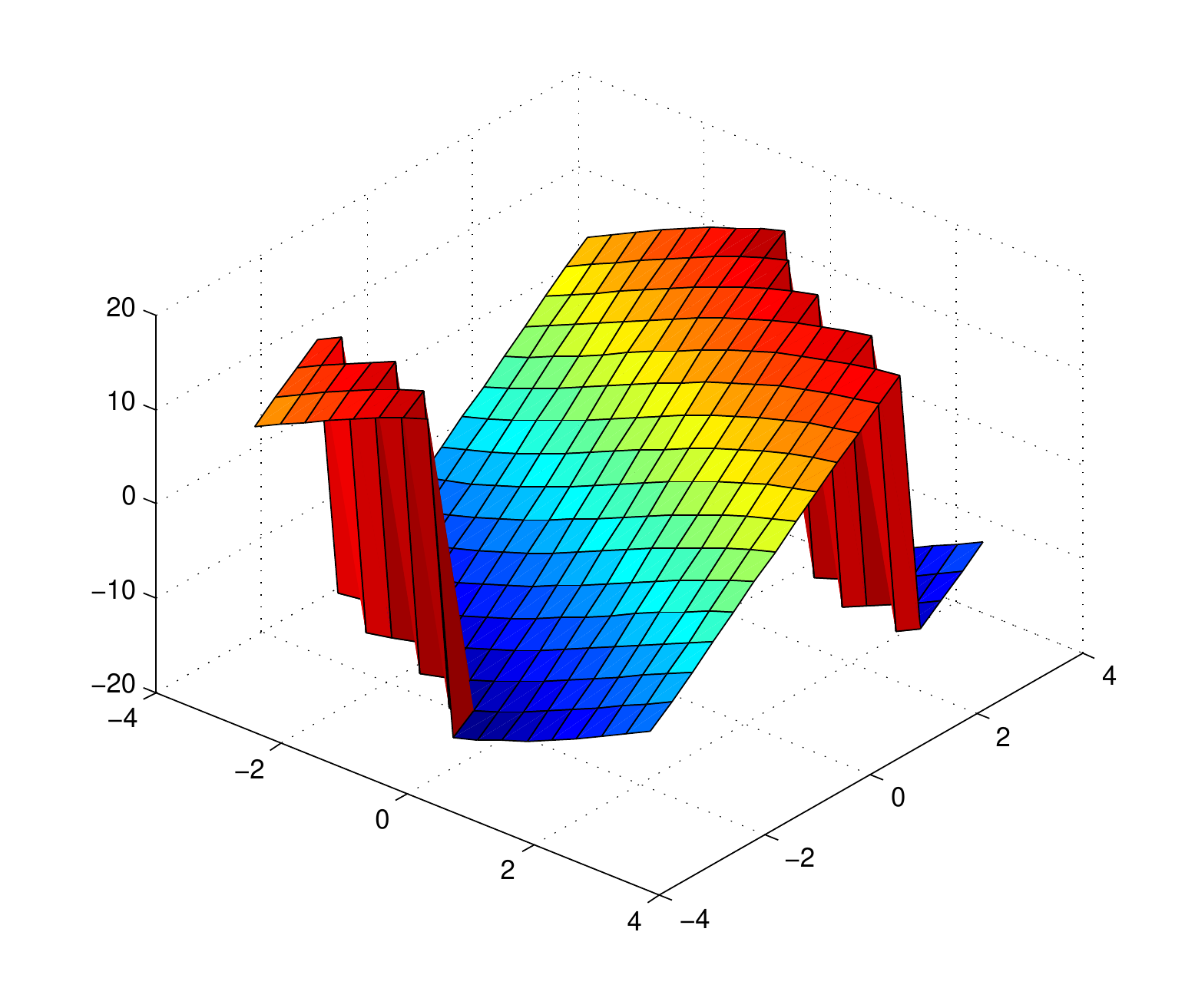} 
    \includegraphics[scale=0.45]{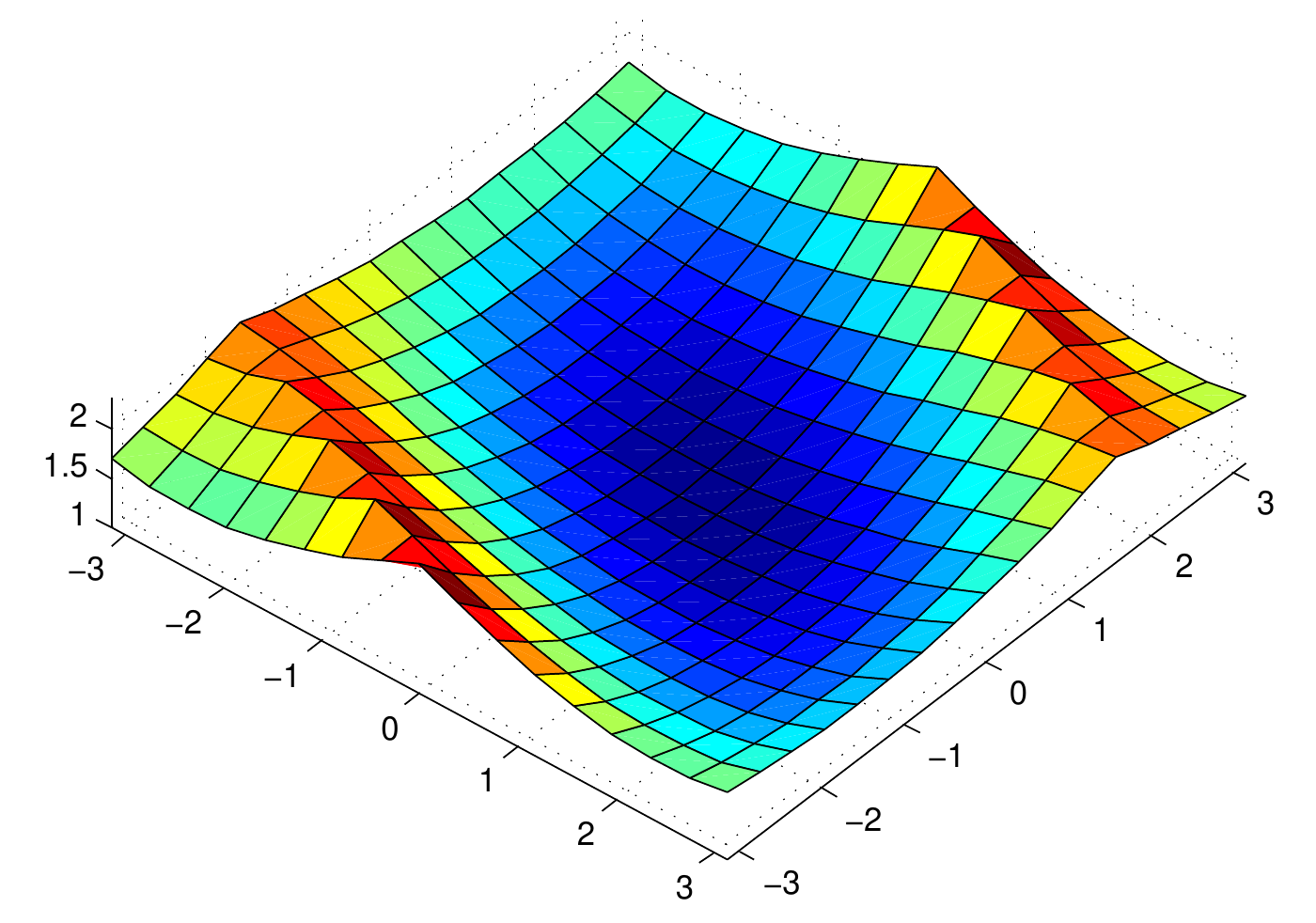}
    \\
    \includegraphics[scale=0.45]{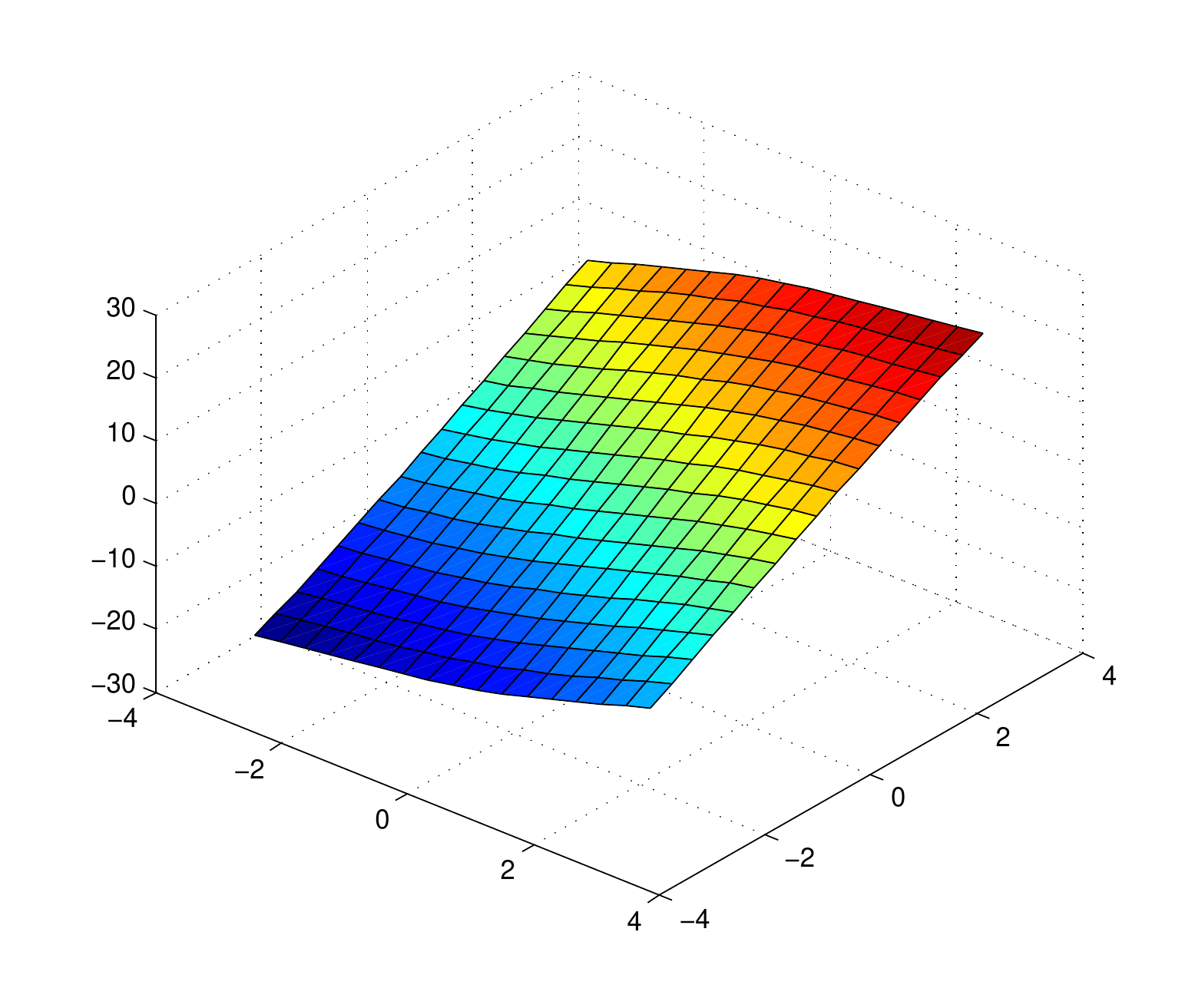}
    \includegraphics[scale=0.45]{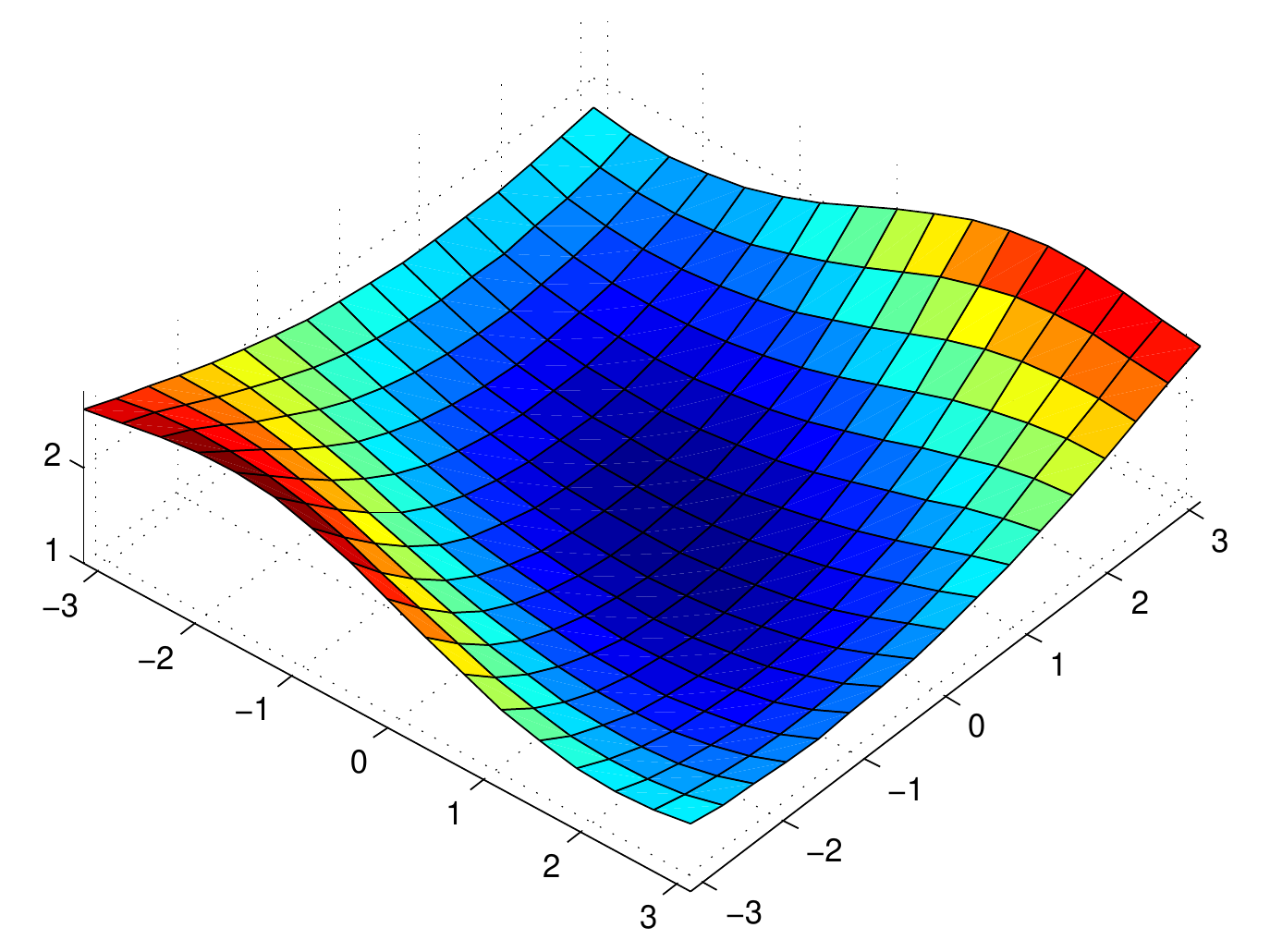}
  \end{center}
  \caption{On the top, computation of parameter $A$ and minimal length and
           for nonlinear system~\eqref{eq:F:nonlin:6}.
           On the bottom computation of parameters $A$ and $L$
           for Problem~\ref{prob:2}.
           Computation of the figures on the top are done using auxiliary 
           function \reffun{alg:buildGrid}.}
  \label{fig:AL}
\end{figure}

\section{Computation of Fresnel and related integrals}\label{sec:5}
Among the various possible definitions, we choose the following one, which is the same adopted in
\citeNP{abramowitz:1964}.
\begin{definition}[Fresnel Integrals]
Following \citeNP{abramowitz:1964} the Fresnel Integrals are defined as
\begin{EQ}\label{fresnel}
  \Cf(t)=\int_0^t\cos \left(\frac{\pi}{2}\tau^2\right)\dtau,\qquad
  \Sf(t)=\int_0^t\sin \left(\frac{\pi}{2}\tau^2\right)\dtau.\qquad
\end{EQ} 
\end{definition}
\begin{remark}
Some authors prefer the definition
\begin{EQ}
  \Cf(t)=\int_0^t\cos(\tau^2)\dtau,
  \quad\textrm{and}\quad
  \Sf(t)=\int_0^t\sin(\tau^2)\dtau.
\end{EQ}
However by using the identities
\begin{EQ}
  \int_0^t\cos(\tau^2)\dtau =\dfrac{\sqrt{\pi}}{\sqrt{2}} 
  \int_0^{\frac{\sqrt{2}}{\sqrt{\pi}}t}\cos\left(\frac{\pi}{2}\tau^2\right)\dtau,
  \qquad
  \int_0^t\sin(\tau^2)\dtau =\dfrac{\sqrt{\pi}}{\sqrt{2}} 
  \int_0^{\frac{\sqrt{2}}{\sqrt{\pi}}t}\sin\left(\frac{\pi}{2}\tau^2\right)\dtau,
\end{EQ}
it is easy to pass between the two definitions.
\end{remark}

To compute the standard Fresnel integrals one can use algorithms described in 
\citeNP{Snyder:1993},  \citeNP{Smith:2011} and \citeNP{Thompson:1997} or using continued fraction 
expansion as in \citeNP{Backeljauw:2009}.
It is possible to reduce the integrals \eqref{clot} 
to a linear combination of this standard Fresnel integrals \eqref{fresnel}.
However simpler expression are obtained using also the momenta of the Fresnel integrals:
\begin{EQ}\label{fresnel:moments}
  \Cf_k(t)=\int_0^t\!\!\!\tau^k\cos \left(\frac{\pi}{2}\tau^2\right)\dtau,\quad
  \Sf_k(t)=\int_0^t\!\!\!\tau^k\sin \left(\frac{\pi}{2}\tau^2\right)\dtau.\qquad
\end{EQ} 
Notice that $\Cf(t)\DEF \Cf_0(t)$ and $\Sf(t)\DEF\Sf_0(t)$.
Closed forms via the exponential integral or the Gamma function are also possible,
however we prefer to express them as a recurrence.  
Integrating by parts, the following recurrence is obtained:
\begin{EQ}[rcl]\label{eq:fresnel:recurrence}
  \Cf_{k+1}(t) & = &
  \dfrac{1}{\pi}\left(
    t^{k}\sin\left(\frac{\pi}{2}t^2\right)-k\,\Sf_{k-1}(t)
    \right),\\
  \Sf_{k+1}(t) & = &
  \dfrac{1}{\pi}\left(
    k\,\Cf_{k-1}(t)-t^{k}\cos\left(\frac{\pi}{2}t^2\right)\right).
\end{EQ} 
Recurrence is started by computing standard Fresnel integrals~\eqref{fresnel}
and (by using the change of variable $z=\tau^2$) the following values
\begin{EQ}
  \Cf_1(t) = \frac{1}{\pi}\sin \left(\frac{\pi}{2}t^2\right),
  \qquad
  \Sf_1(t) = \frac{1}{\pi}\left(1-\cos\left(\frac{\pi}{2}t^2\right)\right),
\end{EQ} 
From recurrence follows that $\Cf_k(t)$ and $\Sf_k(t)$ with $k$ odd 
do not contain Fresnel integrals \eqref{fresnel}
and are combination of elementary functions.

The computation of clothoids relies most on the evaluation of integrals of kind \eqref{eq:def:Theta} with their derivatives. The reduction is possible via a change of variable and integration by parts. It is enough to consider two integrals, that cover all possible cases:
\begin{EQ}[rcl]\label{eq:fresnel:general}
  X_k(a,b,c)&=&\int_0^1 \tau^k\cos\left(\frac{a}{2}\tau^2+b\tau+c\right)\dtau,\\
  Y_k(a,b,c)&=&\int_0^1 \,\tau^k\sin\left(\frac{a}{2}\tau^2+b\tau+c\right)\dtau.
\end{EQ}
In fact, integrals \eqref{eq:g} and \eqref{eq:h} with their
derivatives can be evaluated by using the identity
\begin{EQ}[l]
   \int_0^s \cos\left(\frac{1}{2}\kappa'\tau^2+\kappa\tau+\vartheta_0\right)\dtau=
   s\int_0^1 \cos\left(\dfrac{a}{2}z^2+bz+c\right)\dz
\end{EQ}
where $a= \kappa's^2$, $b =\kappa s$, and $c = \vartheta_0$ so that
\begin{EQ}
   g(\eta) = Y_0(2\eta,\Delta\vartheta-\eta,-\varphi),\qquad
   h(\eta) = X_0(2\eta,\Delta\vartheta-\eta,-\varphi),
\end{EQ}
and finally, equation \eqref{clot} can be evaluated as
\begin{EQ}[rcl]
  x(s) &=& x_0 + s\,X_0(\kappa's^2,\kappa s,\vartheta_0), \\
  y(s) &=& y_0 + s\,Y_0(\kappa's^2,\kappa s,\vartheta_0).
\end{EQ}
From now on, algorithms for accurate computation of $X_k$ and $Y_k$
are discussed.
From the well known identities
\begin{EQ}\label{eq:cos:sin:identity}
  \cos(\alpha+\beta)=\cos\alpha\cos\beta-\sin\alpha\sin\beta,\quad
  \sin(\alpha+\beta)=\sin\alpha\cos\beta+\cos\alpha\sin\beta,\qquad
\end{EQ}
integrals \eqref{eq:fresnel:general} can be rewritten as
\begin{EQ}[rcl]
  X_k(a,b,c)&=&X_k(a,b,0)\cos c-Y_k(a,b,0)\sin c, \\
  Y_k(a,b,c)&=&X_k(a,b,0)\sin c+Y_k(a,b,0)\cos c,
\end{EQ}
Thus defining $X_k(a,b)\DEF X_k(a,b,0)$ and $Y_k(a,b)\DEF Y_k(a,b,0)$
the computation of \eqref{eq:fresnel:general}
is reduced to the computation of $X_k(a,b)$ and $Y_k(a,b)$.
It is convenient to introduce the following quantities 
\begin{EQ}
  s = \dfrac{a}{\abs{a}},\qquad
  z =\frac{\sqrt{\abs{a}}}{\sqrt{\pi}},\qquad
  \ell
  =\frac{s\,b}{\sqrt{\pi\abs{a}}},\quad
  \gamma
  =-\frac{s\,b^2}{2\abs{a}},
\end{EQ}
so that it is possible to rewrite the argument of the trigonometric
functions of $X_k(a,b)$ and $Y_k(a,b)$ as
\begin{EQ}[rcl]
   \frac{a}{2}\tau^2+b\tau&=&
   \dfrac{\pi}{2}s
   \left(\tau \dfrac{\sqrt{\abs{a}}}{\sqrt{\pi}}+\frac{s\,b}{\sqrt{\pi\abs{a}}}\right)^2
   -\frac{s\,b^2}{2\abs{a}}
   =
   \dfrac{\pi}{2}s\big(\tau\,z+\ell\big)^2+\gamma.
\end{EQ}
By using the change of variable $\xi=\tau\,z+\ell$ with inverse $\tau= z^{-1}(\xi-\ell)$
for $X_k(a,b)$ and the identity~\eqref{eq:cos:sin:identity} we have:
\begin{EQ}[rcl]
  X_k(a,b) &=& 
  z^{-1}\int\limits_{\ell}^{\ell+z}
  z^{-k}(\xi-\ell)^k\cos\left(\frac{s\pi}{2}\xi^2+\gamma\right)\dxi,\\
  &=& 
  z^{-k-1}
  \int\limits_{\ell}^{\ell+z}  
  \sum_{j=0}^k \binom{k}{j}\,\xi^j\ell^{k-j}\cos\left(\frac{s\pi}{2}\xi^2+\gamma\right)\dxi,\\
  &=& z^{-k-1}\sum_{j=0}^k \binom{k}{j}\,\ell^{k-j}\left[\cos \gamma \Delta\Cf_j-s\sin\gamma\Delta\Sf_j\right],
\end{EQ}
where
\begin{EQ}\label{eq:def:Cj:Sj}
  \Delta\Cf_j= \Cf_j(\ell+z)-\Cf_j(\ell), \quad \Delta\Sf_j= \Sf_j(\ell+z)-\Sf_j(\ell),
  \qquad
\end{EQ}
are the evaluation of the momenta of the Fresnel integrals as defined in \eqref{fresnel:moments}.
Analogously for $Y_k(a,b)$ we have:
\begin{EQ}
  Y_k(a,b) = 
  z^{-k-1}
  \sum_{j=0}^k \binom{k}{j}\ell^{k-j}
  \left[\sin\gamma\Delta\Cf_j+s\cos \gamma \Delta\Sf_j\right].
\end{EQ}
A cheaper way to compute $X_k(a,b)$ and $Y_k(a,b)$ is via  recurrence
as in \eqref{eq:fresnel:recurrence} for the computation of Fresnel momenta.
Integrating the identity,
\begin{EQ}\label{eq:pre:rec:XY}
  \TotDer[t]{} \left[ t^k \sin\left(\dfrac{a}{2}t^2+bt\right)\right]
  =
  k t^{k-1}\sin\left(\dfrac{a}{2}t^2+bt\right)
  + t^k (at+b)\cos\left(\dfrac{a}{2}t^2+bt\right)
\end{EQ}
we have
\begin{EQ}
  \sin\left(\dfrac{a}{2}+b\right)
  =
  k Y_{k-1}(a,b) +
  a X_{k+1}(a,b) + b X_k(a,b)
\end{EQ}
which can be solved for $X_{k+1}(a,b)$.
Repeating the same argument with $\cos((a/2)t^2+bt)$ 
and solving for $Y_{k+1}(a,b)$ we obtain recurrences
for $X$ and $Y$:
\begin{EQ}[rcl]\label{eq:XY:recurrence}
  X_{k+1}(a,b)&=&
  \dfrac{1}{a}\left(
    \sin\left(\frac{a}{2}+b\right)-b\,X_{k}(a,b)-k\,Y_{k-1}(a,b)
  \right),\qquad
  \\
  Y_{k+1}(a,b)&=&
  \dfrac{1}{a}\left(
    k\,X_{k-1}(a,b)-b\,Y_{k}(a,b)-\cos\left(\frac{a}{2}+b\right)
  \right).\qquad
\end{EQ}
From the identities
\begin{EQ}[rclcl]
   \int_0^1 (a\tau+b)\cos\left(\frac{a}{2}\tau^2+b\tau\right)\dtau
   &=&
   a X_1(a,b) + b X_0(a,b)
   &=&
   \sin\left(\frac{a}{2}+b\right),
   \\
   \int_0^1 (a\tau+b)\sin\left(\frac{a}{2}\tau^2+b\tau\right)\dtau
   &=&
   a Y_1(a,b) + b Y_0(a,b)
   &=&
   1 - \cos\left(\frac{a}{2}+b\right),
\end{EQ}
it follows
\begin{EQ}[rcl]\label{eq:XY:recurrence:base}
   X_1(a,b) &=& \dfrac{1}{a}\left(\sin\left(\frac{a}{2}+b\right)-b X_0(a,b)\right),
   \\
   Y_1(a,b) &=& \dfrac{1}{a}\left(1-\cos\left(\frac{a}{2}+b\right)-b Y_0(a,b)\right).
\end{EQ}
To complete recurrence we need the initial values $X_0(a,b)$, $Y_0(a,b)$:
\begin{EQ}[rcl]\label{eq:XY:recurrence:base:0}
   X_0(a,b) &=& z^{-1}\big(\cos \gamma \Delta\Cf_0-s\sin\gamma\Delta\Sf_0\big), \\
   Y_0(a,b) &=& z^{-1}\big(\sin \gamma \Delta\Cf_0+s\cos\gamma\Delta\Sf_0\big),
\end{EQ}
where $\Cf_0$ and $\Sf_0$ is defined in \eqref{eq:def:Cj:Sj}.
This recurrence may be inaccurate when $\abs{a}$ is small, in fact $z$ appears in the denominator of many fractions.
For this reason for small values of $\abs{a}$ we substitute this recurrence 
with asymptotic expansions.

\section{Accurate computation with small parameters}
\label{sec:6}
When the parameter $a$ is small we use 
identity~\eqref{eq:cos:sin:identity} to 
derive series expansion
\begin{EQ}[rcl]\label{eq:serie:X:asmall}
   X_k(a,b)
   &=&
   \int_0^1\tau^k\cos\left(\frac{a}{2}\tau^2+b\tau\right)\dtau \\
   &=&
   \int_0^1\tau^k\left[
   \cos\left(\frac{a}{2}\tau^2\right)\cos(b\tau)
   -
   \sin\left(\frac{a}{2}\tau^2\right)\sin(b\tau)
   \right]\dtau
   \\
   &=&
   \sum_{n=0}^\infty\frac{(-1)^n}{(2n)!}\left(\frac{a}{2}\right)^{2n}
   X_{4n+k}(0,b)
   -
   \sum_{n=0}^\infty\frac{(-1)^n}{(2n+1)!}\left(\frac{a}{2}\right)^{2n+1}
   Y_{4n+2+k}(0,b)
   \\
   &=&
   \sum_{n=0}^\infty\frac{(-1)^n}{(2n)!}\left(\frac{a}{2}\right)^{2n}
   \left[
   X_{4n+k}(0,b)-\dfrac{a\,Y_{4n+2+k}(0,b)}{2(2n+1)}
   \right]
\end{EQ}
and analogously using again identity~\eqref{eq:cos:sin:identity}
we have the series expansion
\begin{EQA}[rcl]\label{eq:serie:Y:asmall}
   Y_k(a,b) &=&
   \int_0^1\tau^k\sin\left(\frac{a}{2}\tau^2+b\tau\right)\dtau \\
   &=&
   \sum_{n=0}^\infty\frac{(-1)^n}{(2n)!}\left(\frac{a}{2}\right)^{2n}
   \left[
   Y_{4n+k}(0,b)+\dfrac{a\,X_{4n+2+k}(0,b)}{2(2n+1)}.
   \right]
\end{EQA}
From the inequalities
\begin{EQ}
  \abs{X_{k}}\leq \int_{0}^{1}|\tau^{k}|\dtau=\dfrac{1}{k+1},
  \qquad
  \abs{Y_{k}}\leq \int_{0}^{1}|\tau^{k}|\dtau=\dfrac{1}{k+1},
\end{EQ}
we can estimate the remainder for the series of $X_k$
\begin{EQ}[rcl]
   R_{p,k} &=& 
   \abs{
   \sum_{n=p}^\infty\frac{(-1)^n}{(2n)!}\left(\frac{a}{2}\right)^{2n}
   \left[
   X_{4n+k}(0,b)-\dfrac{a\,Y_{4n+2+k}(0,b)}{2(2n+1)}
   \right]}
   \\
   &\leq&
   \sum_{n=p}^\infty\frac{1}{(2n)!}\left(\frac{a}{2}\right)^{2n}
   \left[\dfrac{1}{4n+1}+\dfrac{\abs{a}}{2(2n+1)(4n+3)}
   \right]
   \\
   &\leq&
   \left(\frac{a}{2}\right)^{2p}
   \sum_{n=p}^\infty\frac{1}{(2(n-p))!}\left(\frac{a}{2}\right)^{2(n-p)}
   \\
   &\leq&
   \left(\frac{a}{2}\right)^{2p}
   \sum_{n=0}^\infty\frac{1}{(2n)!}\left(\frac{a}{2}\right)^{2n}
   =
   \left(\frac{a}{2}\right)^{2p}\cosh(a).
\end{EQ}
The same estimate is obtained for the series of $Y_k$
\begin{remark}\label{rem:epsilon}
Series \eqref{eq:serie:X:asmall} and \eqref{eq:serie:Y:asmall}
converge fast. For example if $\abs{a}<10^{-4}$ and $p=2$ the error is less than $6.26\cdot 10^{-18}$
while if $p=3$ the error is less than $1.6\cdot 10^{-26}$.
\end{remark}
Recurrence~\eqref{eq:XY:recurrence}-\eqref{eq:XY:recurrence:base}-\eqref{eq:XY:recurrence:base:0}
permits to compute 
$X_k(a,b)$ and $Y_k(a,b)$ at arbitrary precision when $a\neq 0$, but when $a=0$ it modifies to
\begin{EQ}[rcl]\label{eq:XY:recu:a:small}
   X_k(0,b) & = & b^{-1}\big(\sin b-k\,Y_{k-1}(0,b)\big),\\
   Y_k(0,b) & = & b^{-1}\big(k\,X_{k-1}(0,b)-\cos b\big),
\end{EQ}
with starting point
\begin{EQ}
   X_0(0,b) = b^{-1}\sin b,\qquad
   Y_0(0,b) = b^{-1}(1-\cos b).
\end{EQ}
Recurrence~\eqref{eq:XY:recu:a:small} with~\eqref{eq:serie:X:asmall} and
\eqref{eq:serie:Y:asmall} permits computation when $b\neq 0$,
unfortunately recurrence~\eqref{eq:XY:recu:a:small} is highly
unstable and cannot be used.
In alternative an explicit formula based on Lommel function
$s_{\mu,\nu}(z)$ can be used \cite{Shirley:2003}.
Explicit formula is the following
\begin{EQ}[rcl]\label{eq:int:XY0b}
   X_k(0,b) &=&
   \dfrac{
     k s_{k+\frac{1}{2},\frac{3}{2}}(b)\sin b +f(b)s_{k+\frac{3}{2},\frac{1}{2}}(b)
   }{(1+k)b^{k+\frac{1}{2}}}
   + \dfrac{\cos b}{1+k},\qquad
   \\
   Y_k(0,b) &=&
   \dfrac{
     k s_{k+\frac{3}{2},\frac{3}{2}}(b)\sin b+g(b)s_{k+\frac{1}{2},\frac{1}{2}}(b)
   }{(2+k)b^{k+\frac{1}{2}}}
   + \dfrac{\sin b}{2+k},
\end{EQ}
where $f(b) = b^{-1}\sin b-\cos b$ and $g(b)=f(b)(2+k)$. Lommel function
has the following expansion (see~\url{http://dlmf.nist.gov/11.9})
\begin{EQ}\label{eq:lommel}
  s_{\mu,\nu}(z) 
  =
  z^{\mu+1}\sum_{n=0}^\infty\dfrac{(-z^2)^n}{\alpha_{n+1}(\mu,\nu)},
  \qquad
  \alpha_n(\mu,\nu)=\prod_{m=1}^n ((\mu+2m-1)^2-\nu^2),
\end{EQ}
and using this expansion in~\eqref{eq:int:XY0b} results in the following
explicit formula
\begin{EQ}[rcl]
   X_k(0,b) &=&
   A(b)w_{k+\frac{1}{2},\frac{3}{2}}(b)+
   B(b)w_{k+\frac{3}{2},\frac{1}{2}}(b)+
   \dfrac{\cos b}{1+k},\qquad
   \\
   Y_k(0,b) &=&
   C(b)w_{k+\frac{3}{2},\frac{3}{2}}(b)+
   D(b)w_{k+\frac{1}{2},\frac{1}{2}}(b)+
   \dfrac{\sin b}{2+k},\qquad
\end{EQ}
where
\begin{EQ}
   w_{\mu,\nu}(b) =
   \sum_{n=0}^\infty\dfrac{(-b^2)^n}{\alpha_{n+1}(\mu,\nu)}, \quad
   A(b) = \dfrac{kb\sin b}{1+k},\quad
   B(b) = \dfrac{(\sin b-b\cos b)b}{1+k},\\
   C(b) = -\dfrac{b^2\sin b}{2+k},\quad
   D(b) = \sin b-b\cos b.
\end{EQ}

\section{Concluding remarks}
The proposed algorithm is detailed in the appendix using pseudocode
and can be easily translated in any programming language.
A MATLAB implementation is furnished to test the functions
introduced in the paper.
It is important that the computation of Fresnel integrals
is accurate because they are involved in many expansions
for the clothoid fitting.
In the MATLAB implementation of the proposed algorithm
for Fresnel integrals approximation a slightly modified version
of Venkata Sivakanth Telasula script, available in MATLAB Central,
was used.

\appendix
\section{Algorithms for the computation of Fresnel related integrals}
We present here the algorithmic version of the analytical expression we derived in 
Section~\ref{sec:5} and~\ref{sec:6}.
This algorithm are necessary for the computation of 
the main function \reffun{alg:buildClothoid} of Section~\ref{sec:4} which takes the input data ($x_0$, $y_0$, $\vartheta_0$, $x_1$, $y_1$, $\vartheta_1$) and returns the parameters ($\kappa$, $\kappa'$, $L$) that solve the problem as expressed in equation \eqref{clot}.
Function \reffun{alg:evalXY}
computes the generalized Fresnel integrals \eqref{eq:fresnel:general}.
It distinguishes the cases of $a$ larger or smaller than a threshold $\varepsilon$.
The value of  $\varepsilon$ is discussed in Section~\ref{sec:6},
see for example Remark \ref{rem:epsilon}. 
Formulas~\eqref{eq:XY:recurrence}-\eqref{eq:XY:recurrence:base}-\eqref{eq:XY:recurrence:base:0},
used to compute $X_k(a,b)$ and $Y_k(a,b)$ at arbitrary precision when 
$\abs{a}\geq\varepsilon$, are implemented in function \reffun{alg:evalXYaLarge}.
Formulas~\eqref{eq:serie:X:asmall}-\eqref{eq:serie:Y:asmall},
used to compute $X_k(a,b)$ and $Y_k(a,b)$ at arbitrary precision when 
$\abs{a}<\varepsilon$, are implemented in function~\reffun{alg:evalXYaSmall}.
This function requires computation of~\eqref{eq:XY:recu:a:small} and \eqref{eq:int:XY0b}
implemented in function \reffun{alg:evalXYaZero}
which needs (reduced) Lommel function~\eqref{eq:lommel} implemented in 
function \reffun{alg:S}.

\begin{function}[!htcb]
  \caption{evalXY($a$, $b$, $c$, $k$)}
  \label{alg:evalXY}
  \def\assign{\leftarrow}
  \small
  \SetAlgoLined
  \lIf{$\abs{a}<\varepsilon$}{
    $X^0,Y^0\assign$\ref{alg:evalXYaSmall}($a$,$b$,$k$,$p$);
  }
  \lElse{
    $X^0,Y^0\assign$\ref{alg:evalXYaLarge}($a$,$b$,$k$)\;
  }
  \For{$j=0,1,\ldots,k$}{
    $\ASSIGNs{X_{j}}X_j^0\cos c-Y_j^0\sin c$;\quad
    $\ASSIGNs{Y_{j}}X_j^0\sin c+Y_j^0\cos c$\;
  }
  \Return{$X$, $Y$}
\end{function}

\begin{function}[!htcb]
  \caption{evalXYaLarge($a$, $b$, $k$)}
  \label{alg:evalXYaLarge}
  \def\assign{\leftarrow}
  \small
  \SetAlgoLined
  $s\assign\dfrac{a}{\abs{a}}$;\quad
  $z\assign\dfrac{\sqrt{\pi}}{\sqrt{\abs{a}}}$;\quad
  $\ell\assign\dfrac{s\,b}{z\,\pi}$;\quad
  $\gamma\assign-\dfrac{sb^2}{2\abs{a}}$;\quad
  $t\assign\dfrac{1}{2}a+b$\;
  $\ASSIGNl{\Delta\Cf_0}\Cf(\ell+z)-\Cf(\ell)$\;
  $\ASSIGNl{\Delta\Sf_0}\Sf(\ell+z)-\Sf(\ell)$\;
  $\ASSIGNl{X_0}z^{-1}\left(\cos\gamma\,\Delta\Cf_0-s\,\sin\gamma\,\Delta\Sf_0\right)$\;
  $\ASSIGNl{Y_0}z^{-1}\left(\sin\gamma\,\Delta\Cf_0+s\,\cos\gamma\,\Delta\Sf_0\right)$\;
  $\ASSIGNl{X_1}a^{-1}\left(\sin t-b\,X_0\right)$\;
  $\ASSIGNl{Y_1}a^{-1}\left(1-\cos t-b\,Y_0\right)$\;
  \BlankLine
  \For{$j=1,2,\ldots,k-1$}{
    $\ASSIGNl{X_{j+1}}a^{-1}(\sin t-b\,X_{j}-j\,Y_{j-1})$\;
    $\ASSIGNl{Y_{j+1}}a^{-1}(j\,X_{j-1}-b\,Y_{j}-\cos t)$\;
  }
  \Return{$X$, $Y$}
\end{function}

\begin{function}[!htcb]
  \caption{rLommel($\mu$, $\nu$, $b$)}
  \label{alg:S}
  \SetKw{AND}{and}
  \SetKw{BREAK}{break}
  \def\assign{\leftarrow}
  \small
  \SetAlgoLined

  $\ASSIGNs{t}(\mu+\nu+1)^{-1}(\mu-\nu+1)^{-1}$;\quad
  $\ASSIGNs{r}t$;\quad$n\assign 1$\;
  \lWhile{$\abs{t}>\varepsilon$}{
    $t\assign t \dfrac{(-b)}{2n+\mu-\nu+1}\dfrac{b}{2n+\mu+\nu+1}$;\quad
    $r\assign r+t$;\quad
    $n\assign n+1$\;
  }
  \Return{$r$}
\end{function}

\begin{function}[!htcb]
  \caption{evalXYaZero($b$, $k$)}
  \label{alg:evalXYaZero}
  \def\assign{\leftarrow}
  \small
  \SetAlgoLined
  \uIf{$\abs{b}<\varepsilon$}{
    $X_0\assign1-\dfrac{b^2}{6}\left(1-\dfrac{b^2}{20}\right)$;\quad
    $Y_0\assign\dfrac{b^2}{2}\left(1-\dfrac{b^2}{6}\left(1-\dfrac{b^2}{30}\right)\right)$\;
  }
  \Else{
    $\ASSIGNs{X_0}\dfrac{\sin b}{b}$;\quad
    $\ASSIGNs{Y_0}\dfrac{1-\cos b}{b}$\;
  } 
  $\ASSIGNs{A}b\sin b$;\quad
  $\ASSIGNs{D}\sin b-b\cos b$;\quad
  $\ASSIGNs{B}bD$;\quad
  $\ASSIGNs{C}-b^2\sin b$\;
  \For{$k=0,1,\ldots,k-1$}{
    $\ASSIGNl{X_{k+1}}\dfrac{kA\,\reffun{alg:S}\big(k+\frac12,\frac32,b\big) +
                         B\,\reffun{alg:S}\big(k+\frac32,\frac12,b\big) + \cos b}{1+k}$\;
    $\ASSIGNl{Y_{k+1}}\dfrac{C\,\reffun{alg:S}\big(k+\frac32,\frac32,b\big) + \sin b}{2+k}
     + D\,\reffun{alg:S}\Big(k+\frac12,\frac12,b\Big)$\;
  }
  \Return{$X$, $Y$}
\end{function}

\begin{function}[!htcb]
  \caption{evalXYaSmall($a$, $b$, $k$, $p$)}
  \label{alg:evalXYaSmall}
  \SetKwFunction{findA}{findA}
  \SetKwFunction{Aguess}{Aguess}
  \SetKwFunction{normalizeAngle}{normalizeAngle}
  \def\assign{\leftarrow}
  \small
  \SetAlgoLined
  $X^0,Y^0\assign$~\ref{alg:evalXYaZero}$(b, k+4p+2)$;\quad
  $t\assign 1$\;
  \lFor{$j=0,1,\ldots,k$}{
    $\ASSIGNs{X_j}X^0_j-\dfrac{a}{2}Y^0_{j+2}$;\quad
    $\ASSIGNs{Y_j}Y^0_j+\dfrac{a}{2}X^0_{j+2}$\;
  }
  \For{$n=1,2,\ldots,p$}{
    $t\assign\dfrac{-t\,a^2}{16n(2n-1)}$\;
    \For{$j=0,1,\ldots,k$}{
      $\ASSIGNs{X_j}X_j + t\dfrac{X^0_{4n+j}-aY^0_{4n+j+2}}{4n+2}$;\quad
      $\ASSIGNs{Y_j}Y_j + t\dfrac{Y^0_{4n+j}+aX^0_{4n+j+2}}{4n+2}$\;
    }
  }
  \Return{$X$, $Y$}
\end{function}

\bibliographystyle{elsarticle-harv} 
\bibliography{Clothoid_Spirals-refs} 


\end{document}